\documentclass[12pt]{article}
\usepackage{amssymb,amsthm}

\newtheorem{thm}{Theorem}

\def\card#1{\vert #1 \vert}
\def\gpindex#1#2{\card {#1\colon #2}}
\def\irr#1{{\rm  Irr}(#1)}

\def\cent#1#2{{\bf C}_{#1}(#2)}
\def\gpcen#1{{\bf Z} (#1)}
\def\ker#1{{\rm ker} (#1)}

\begin{document}

\title{Classifying Camina groups: \\ A theorem of Dark and Scoppola}

\author {
       Mark L.\ Lewis
    \\ {\it Department of Mathematical Sciences, Kent State University}
    \\ {\it Kent, Ohio 44242}
    \\ E-mail: lewis@math.kent.edu
       }
\date{September 26, 2011}

\maketitle

\begin{abstract}
Recall that a group $G$ is a Camina group if every nonlinear
irreducible character of $G$ vanishes on $G \setminus G'$. Dark and
Scoppola classified the Camina groups that can occur.  We present a
different proof of this classification using Theorem \ref{lem},
which strengthens a result of Isaacs on Camina pairs. Theorem
\ref{lem} is of independent interest.

MSC Primary: 20D10, MSC Secondary: 20C15

\end{abstract}


\section{Introduction}

Throughout this note $G$ will be a finite group.  In this note, we focus on Camina groups.  A nonabelian group $G$ is a Camina group if the conjugacy class of every element $g \in G \setminus G'$ is $gG'$.  It is clear from the earliest papers that the study of Camina groups and the more general objects Camina pairs was motivated by finding a common generalization of Frobenius groups and extra-special groups.  It is not difficult to see that extra-special groups and Frobenius groups with an abelian Frobenius complement are Camina groups.  Thus, one question that seems reasonable to ask is whether there are any other Camina groups and if we can classify the Camina groups.

In \cite{DaSc}, Dark and Scoppola stated that they had  completed the classification of all Camina groups.  The classification is encoded in the following theorem.

\begin{thm}{\rm [Dark and Scoppola]}\label{main}
Let $G$ be a group.  Then $G$ is Camina group if and only if one of the following holds:

\begin{enumerate}
\item $G$ is a Camina $p$-group of nilpotence class $2$ or $3$.
\item $G$ is a Frobenius group with a cyclic Frobenius complement.
\item $G$ is a Frobenius group whose Frobenius complement is isomorphic to the quaternions.
\end{enumerate}
\end{thm}

In fact, the work in \cite{DaSc} is the capstone of several results that combined lead to the classification.  The first result needed is that if $G$ is a Camina group, $P$ is a Sylow $p$-subgroup for some prime $p$, and $G/G'$ is a $p$-group, then $P$ is a Camina group which is proved in Lemma 3.6 of \cite{ChMc}.  The second result needed is Theorem 3 of \cite{ChMaSc} which states that if $G$ is a Camina group such that $G/G'$ is a $p$-group for some prime $p$ and a Sylow $p$-subgroup of $G$ has nilpotence class at most $p + 1$, then either $G$ is a $p$-group or $G$ is a Frobenius group whose complement is either cyclic or quaternion.  Then Dark and Scopolla proved in \cite{DaSc} that Camina $p$-groups have nilpotence class $2$ or $3$.  This proves Theorem \ref{main} under the hypothesis that $G/G'$ is a $p$-group for some prime $p$.  Finally, by Theorem 2.1 of \cite{coprime}, we know that if $G$ is a Camina group, then either $G$ is a Frobenius group with abelian Frobenius complement or $G/G'$ is a $p$-group for some prime $p$.  Combining all of these results, one obtains the above theorem.

The referee has pointed out that there is a fixable gap in the argument in the previous paragraph.  Theorem 3 of \cite{ChMaSc} relies on Lemma 2.1 of \cite{more}.  As stated in \cite{DaSc}, ``the sketch of proof given for Lemma 2.1 in \cite{more} does not appear to be conclusive.''  The referee has then stated that ``it is easy to fix the gap in the proof of Theorem 3 of \cite{ChMaSc} and obtain the corollary of \cite{DaSc}, because the hypothesis that $P$ has class $3$, with the fact that the center of $P$ is cyclic, cuts out the part of the proof that uses Lemma 2.1 of \cite{more}.''  In addition, we should mention that our Ph.D. student Nabil Mlaiki has recently proved in \cite{diss} that Lemma 2.1 of \cite{more} is true, and this would give a different way of mending the argument.

Our purpose in this paper is to present a different proof of most of Theorem \ref{main}.  In particular, we prove:

\begin{thm}\label{ours}
If $G$ is Camina group, then one of the following holds:
\begin{enumerate}
\item
$G$ is a $p$-group, or
\item
$G$ is a Frobenius group whose complement is
either cyclic or quaternion.
\end{enumerate}
\end{thm}

In other words, we do not prove that Camina $p$-groups have nilpotence class at most $3$.  We refer the reader to Dark and
Scoppola, \cite{DaSc}, for the proof that Camina $p$-groups have nilpotence class $2$ or $3$, and in fact, we will use that result in our work here.  Our proof of Theorem \ref{ours} is based on the work of Isaacs regarding Camina pairs in \cite{coprime}.  The key to our work is to generalize Lemma 3.1 of \cite{coprime} where Isaacs proved that if $P$ is a $p$-group with class at most $2$ and $P$ acts on a nontrivial $p'$-group $Q$ so that $\cent Px \le P'$ for all $x \in Q \setminus \{ 1 \}$, then the action is Frobenius and $P$ is either cyclic or isomorphic to the quaternions.  Theorem \ref{lem} builds on this result. We have replaced the hypothesis that the nilpotence class is at most $2$ with the hypothesis that $P$ is a Camina group. Essentially, what we prove is that $P$ cannot be a Camina $p$-group with nilpotence class $3$.

\begin{thm}\label{lem}
Let $P$ be a Camina $p$-group that acts on a nontrivial $p'$-group $Q$ so that $\cent Px \le P'$ for every $x \in Q \setminus \{ 1 \}$.  Then the action of $P$ is Frobenius and $P$ is the quaternions.
\end{thm}

If $P$ and $Q$ satisfy all the hypotheses of Theorem \ref{lem} except the first (i.e., we are not assuming that $P$ is a Camina group), then $P$ is called a Frobenius-Wielandt complement.  For $2$-groups, the generalized quaternion groups are such groups in their Frobenius action.  An example of such a group is constructed for each odd prime in Beispiel III.10.15 of \cite{hup}.  These groups have been studied in a number of places (in particular, see \cite{esp} and \cite{scop}).  In any case, some additional hypothesis is needed on $P$ for the conclusion Theorem \ref{lem} to be true.


Most of the work in proving the Theorem \ref{ours} is included in the proof of Theorem \ref{lem}.
One motivation for publishing this note is that Herzog, Longobardi, and Maj make use of Theorem \ref{lem} in their article \cite{infinite}.
%
%
%
We would like to thank David Chillag, Marcel Herzog, Marty
Isaacs, and the referee for their useful comments on this paper.

\section{Proof}

In proving Theorem \ref{lem}, we will strongly use facts regarding Camina $p$-groups that were proved in \cite{DaSc}, \cite{MacD1}, \cite{more}, and \cite{Mann}.

\begin{proof}[Proof of Theorem \ref{lem}]
We will assume that the lemma is not true, and work to find a
contradiction.  We take $P$ to be a group that violates the
conclusion with $\card P$ minimal.

If $P$ has nilpotence class at most $2$, then our conclusion is the
conclusion of Lemma 3.1 of \cite{coprime}, and so, $P$ is not a
counterexample. Thus, $P$ has nilpotence class at least $3$.
MacDonald proved in \cite{more} that a Camina $2$-group has
nilpotence class $2$, and $p$ must be odd.  Dark and Scoppola proved
in \cite{DaSc} that a Camina $p$-group has nilpotence class at most
$3$. Therefore, $P$ must have nilpotence class $3$.

Note that we may assume that $Q$ has no proper nontrivial subgroups
that are invariant under the action of $P$ since we may replace $Q$
by any such subgroup.  Since $Q$ has no proper nontrivial subgroups
that are invariant under the action of $P$, it follows that $Q$ is
an elementary abelian $q$-group for some prime $q \ne p$.  Let $Z =
\gpcen P$. Suppose that there exists some element $x \in Q \setminus
\{ 1 \}$ with $\cent Px \cap Z > 1$. Let $Y = \cent Px \cap Z$, and
let $C = \cent QY$.  We have $x \in C$, and since $Y$ is normal in
$P$, it follows that $P$ stabilizes $C$.  This implies that $C = Q$.
Hence, $P/Y$ acts on $Q$.  Notice that $\cent {P/Y}x = \cent Px/Y
\le P'/Y$ for every $x \in Q \setminus \{ 1 \}$. Also, since $Y$ is
central and $P$ has nilpotence class $3$, it follows that $P/Y$ is
not abelian, and hence, $P/Y$ is a Camina $p$-group. Applying the
inductive hypothesis, we have that $P/Y$ is the quaternions. Since
$p$ is odd, this is a contradiction.

We now have $\cent Px \cap Z = 1$ for all $x \in Q \setminus \{ 1
\}$. This implies that $Z$ acts Frobeniusly on $Q$. Hence, $Z$ is
cyclic.  Hence, we may view $Q$ as an irreducible module for $F[P]$
where $F$ is the field of $q$ elements.  Let $K$ be the ring of
$F[P]$-endomorphisms of $Q$, and by Schur's lemma $K$ is a division
ring.  It is finite, so $K$ is a field, and we can view $Q$ as a
$K[P]$-module.  (Notice that we are not changing $Q$ as a set, we
just view it differently.) Now, $Q$ will be absolutely irreducible
as a $K[P]$-module.  By Theorem 9.14 of \cite{text}, $Q$ gives rise
to an irreducible representation of algebraically closed field that
contains $K$. Hence, $Q$ will determine an irreducible $q$-Brauer
character for $G$ (see page 264 of \cite{text}).  Since $q$ does not
divide $\card P$, the irreducible $q$-Brauer characters of $P$ are
just the ordinary irreducible characters of $P$ (Theorem 15.13 of
\cite{text}).  Thus, $Q$ corresponds to a complex irreducible
character $\chi$ of $P$. Notice that $\ker {\chi} = \cent PQ$ is
contained in the centralizer of all elements of $Q$, so $\ker {\chi}
\cap Z = 1$, and this implies that $\ker {\chi} = 1$.  Since $G$ is
nonabelian, this implies that $\chi (1) \ge p$.

Now, $P$ is an $M$-group, so there is a subgroup $R$ and a linear
character $\lambda \in \irr R$ so that $\lambda^P = \chi$.  Also,
there is a $K[R]$-module $W$ corresponding to $\lambda$ so that $W^P
= Q$.  Since $\lambda$ is linear, it follows that $R/\ker {\lambda}
= R/ \cent RW$ is cyclic.  We can write $Q = W_1 \oplus W_2 \oplus
\cdots \oplus W_r$ where $r = \gpindex PR = \chi (1)$ and $W_1 \cong
W$.  In \cite{nondiv}, we saw that irreducible characters of Camina
$p$-groups with nilpotence class $3$ whose kernels do not contain
the center are fully-ramified with respect to the center.  It
follows that $\chi$ is fully-ramified with respect to $P/Z$.

We claim that it suffices to find an element $g \in P \setminus RP'$
so that $g^p = 1$.  Suppose such an element $g$ exists.  We can
relabel the $W_i$ so that for $i = \{ 1, \dots, p \}$, we have $W_i
= W_1 g^{i-1}$. Fix $w_1 \in W_1 \setminus \{ 0 \}$, and set $w_i =
w_1 g^{i-1}$.  It follows that $w_i g = w_{i+1}$ for $1 \le i \le
p-1$ and $w_p g = w_1$.  Now, take $x \in Q$ so that $x = (w_1,
\dots, w_p, 0, \dots, 0)$.  Notice that $x \ne 0$ and $g$ just
permutes the $w_i$'s, so $g$ will centralize $x$.  Since $g$ is not
in $P'$, this will violate $\cent Px \le P'$.

Before we work to find the element $x$, we gather some information
regarding $R$.  Notice that $\cent RW$ is contained in the
centralizers of elements of $Q$ whose only nonzero component lies in
$W$.  It follows that $\cent RW \le P'$ and $\cent RW \cap Z = 1$.
Since $\chi$ is induced from $R$, we have $Z \le R$.

We know from \cite{MacD1} that $\gpindex P{P'} = \gpindex {P'}Z^2$
and so, $\gpindex PZ = \gpindex {P'}Z^3$.  Since $\chi$ is
fully-ramified with respect to $P/Z$,  it follows that $\gpindex PR
= \chi (1) = \gpindex PZ^{1/2} = \gpindex {P'}Z^{3/2}$. We also have
$\gpindex {P'}Z^2 = \gpindex P{P'}$. This implies $\gpindex PR =
\gpindex P{P'}^{3/4}$.  By \cite{MacD1}, $P/P'$ is elementary
abelian. Since $RP'/P'$ is cyclic, we conclude that $\gpindex
{RP'}{P'} = p$.  We now have that $\gpindex PR \ge \gpindex P{RP'} =
\gpindex P{P'}/p$ and hence, $\gpindex P{P'} \le p^4$.  {}From
\cite{MacD1}, we know that $\gpindex {P'}Z \ge p^2$, so $\gpindex
P{P'} \ge p^4$.  We conclude that $\gpindex P{P'} = p^4$.  This
implies that $\gpindex PR = p^3 = \gpindex P{RP'}$, and hence, $P'
\le R$ and $\gpindex R{P'} = p$. By \cite{Mann}, we know that $P'$
is elementary abelian. Since $P'/\cent RW$ is a subgroup of $R/\cent
RW$, it is cyclic and hence, has order $p$.  It follows that
$R/\cent RW$ is cyclic of order at least $p^2$.

Since $R/P'$ is cyclic and $P'/Z$ is central in $P/Z$, so $R/Z$ is
cyclic-by-central. This implies that $R/Z$ is abelian. Hence, $R'
\le Z$.  We also know that $R' \le \cent RW$. We conclude that $R' =
1$, and $R$ is abelian.  Since $P' \le R$, this implies that $RP' =
R$ and $R$ centralizes $P'$.  Recall that $P'$ is elementary
abelian, and we saw in the previous paragraph that $R$ does not have
exponent $p$. Since $R$ is abelian, the set $\Omega_1 (R)$ of
elements of order $p$ in $R$ is a subgroup of $R$, and $P' \le
\Omega_1 (R) < R$.   We saw that $\gpindex R{P'} = p$, so $\Omega_1
(R) = P'$.  It follows that the elements in $R \setminus P'$ all
have order at least $p^2$.

We now work to obtain the element $x \in P \setminus R$ so that $x^p
= 1$.  (Notice that $RP' = R$, so this is the previous statement.)
We take an element $a \in P \setminus \cent P{P'}$.  Notice that $R
\le \cent P{P'}$, so if $a^p = 1$, then we take $x = a$.  Hence, we
may suppose that $a^p \ne 1$.  By \cite{Mann}, we know that $P/Z$
has exponent $p$, so $a^p = z \in Z$.  Since $P$ is a Camina group
and $a \in P \setminus P'$, it follows that the conjugacy class of
$a$ is $aP'$.  This implies that $\gpindex P{\cent Pa} = \card
{P'}$, and so, $\card {\cent Pa} = \gpindex P{P'} = p^4$.  We know
that $P'$ is not centralized by $a$, so $\cent {P'}a < P'$.  Now,
since $Z$ is cyclic and $P'$ is elementary abelian, we deduce that
$\card Z = p$.  Also, since $\gpindex P{P'} = p^4$ and $\gpindex
P{P'} = \gpindex {P'}Z^2$, we deduce that $\gpindex {P'}Z = p^2$,
and hence, $\card {P'} = p^3$.  In other words, $\card {\cent {P'}a}
\le p^2$, and $\card {\langle a \rangle \cent {P'}a} \le p^3$.
Therefore, there exists an element $b \in \cent Pa \setminus \langle
a \rangle \cent {P'}a$.

First, suppose that $b^p = 1$.  If $b \in R$, then $b \in \Omega_1
(R) = P'$, and so, $b \in \cent Pa \cap P' = \cent {P'}a$ which is a
contradiction.  Thus, $b \in P \setminus R$, and we take $x = b$.
Hence, we may assume that $b^p \ne 1$.  We know that $b^p \in Z$, so
$b^p = z^i$ for some integer $1 \le i \le p-1$.  There is an integer
$j$ so that $ij \equiv -1 ~({\rm mod~}p)$.  We take $x = ab^j$.
Notice that $x \in \cent Pa$ and $x^p = (ab^j)^p = a^p b^{jp} = z
z^{ij} = z z^{-1} = 1$.  If $x \in R$, then $x \in \Omega_1 (R) =
P'$, and hence, $x \in \cent {P'}a$.  This implies that $ab^j \in
\cent {P'}a$, and thus, $b^j \in a^{-1} \cent {P'}a$.  Since $j$ is
coprime to $p$, we end up with $b \in \langle a \rangle \cent {P'}a$
which is a contradiction.  Therefore, $x \in P \setminus R$, and we
have found the desired element.
\end{proof}

In Theorem 2.1 of \cite{coprime}, Isaacs proved that if $(G,K)$ is a
Camina pair where $G/K$ is nilpotent, then either $G$ is a Frobenius
group with Frobenius kernel $K$ or $G/K$ is a $p$-group for some
prime $p$, and if $G/K$ is a $p$-group, then $G$ has a normal
$p$-complement $M$ and $\cent Gm \le K$ for all $m \in M \setminus
\{ 1 \}$.  We do not define Camina pairs in this paper; however, for
our purposes it is enough to know that $G$ is a Camina group if and
only if $(G,G')$ is a Camina pair.  In \cite{Yong}, Y. Ren has noted
that together Theorem 2.1 of \cite{coprime} with Lemma 3.1 of
\cite{coprime} imply that if $G$ is a Camina group, then one of the
following occurs:
\begin{enumerate}
\item $G$ is Frobenius group with $G'$ its Frobenius kernel.
\item $G$ is a Camina $p$-group.
\item $G = RP$ where $P$ is a Sylow $p$-subgroup of $G$ for some
prime $p$ and $R$ is a normal $p$-complement for $G$ with $R < G'$.
In addition, if $P$ has nilpotence class $2$, then $P$ is the
quaternions and $G$ is a Frobenius group whose Frobenius kernel is
$R$.
\end{enumerate}
Essentially, we will prove using Theorem \ref{lem} that if (3)
occurs, then $P$ must have nilpotence class $2$.  This observation
underlies our argument.

\begin{proof}[Proof of Theorem \ref{ours}]
It is easy to see that each of the groups mentioned are Camina
groups. Thus, we will assume that $G$ is a Camina group, and prove
that it is one of the groups listed.  As we have seen, Theorem 2.1
of \cite{coprime} can be restated in terms of Camina groups as
saying that if $G$ is a Camina group, then either $G$ is a Frobenius
group with $G'$ as its Frobenius kernel or $G/G'$ is a $p$-group for
some prime $p$ and $G$ has a normal $p$-complement $Q$ where $\cent
Gx \le G'$ for all $x \in Q \setminus \{ 1 \}$.

If $G$ is a Frobenius group with $G'$ as its Frobenius kernel, then
the Frobenius complements for $G$ must be abelian, and hence,
cyclic.  Thus, we may assume that $G/G'$ is a $p$-group for some
prime $p$.  If $G$ is a $p$-group, then we are done since we know
that Camina $p$-groups have nilpotence class $2$ or $3$ by
\cite{DaSc}. Thus, we may assume that $G$ is not a $p$-group, and
so, $G$ has a nontrivial normal $p$-complement $Q$. Let $P$ be a
Sylow $p$-subgroup of $G$. Notice that $G' = P'Q$, and so, $P \cap
G' = P'$.  We have $\cent Px \le P \cap G' = P'$ for all $x \in Q
\setminus \{ 1 \}$.  Notice that $P \cong G/Q$ is either abelian or
a Camina group.  If $P$ is abelian, we have $\cent Px = 1$ for all
$x \in Q \setminus \{ 1\}$ and $Q = G'$.  In particular, $G$ is a
Frobenius group with $G'$ as its Frobenius group, and we are done as
before. Finally, we have the case where $P$ is a Camina group where
$\cent Px \le P'$ for all $x \in Q \setminus \{ 1 \}$.  We now apply
Theorem \ref{lem} to see that $P$ acts Frobeniusly on $Q$ and is the
quaternions. Therefore, $G$ is a Frobenius group with Frobenius
kernel $Q$ and Frobenius complement $P$ where $P$ is isomorphic to
the quaternions.
\end{proof}

\end{document}